\documentclass[final]{siamltex}

\usepackage{amsfonts}
\usepackage{amssymb}
\usepackage{amsmath}
\usepackage[a4paper]{geometry}
\usepackage{a4wide}
\usepackage[parfill]{parskip}
\usepackage{color}
\usepackage{url}
\usepackage{graphicx}
\usepackage{xspace}

\newtheorem{observation}{Observation}

\newcommand{\cX}{\mathcal{X}}
\newcommand{\cF}{\mathcal{F}}

\newcommand{\cM}{\mathcal{M}}
\newcommand{\cA}{\mathcal{A}}

\newcommand{\maf}{{\sc MAF}\xspace}
\newcommand{\maaf}{{\sc MAAF}\xspace}
\newcommand{\dfvs}{{\sc DFVS}\xspace}
\newcommand{\poly}{\text{poly}}

\title{Approximation algorithms for nonbinary\\ agreement forests}

\footnotetext[1]{Van Iersel is affiliated with the Centrum Wiskunde \& Informatica (CWI), P.O. Box 94079,
1090 GB Amsterdam, The Netherlands, \tt{l.j.j.v.iersel@gmail.com}.}

\footnotetext[2]{Kelk and Leki\'{c} are afilliated with the Department of Knowledge Engineering (DKE), Maastricht University, P.O. Box 616, 6200 MD Maastricht, The Netherlands, \tt{\{steven.kelk, nela.lekic\}@maastrichtuniversity.nl}.}

\footnotetext[3]{Stougie is affiliated with both CWI and Operations Research, Dept. of Economics and Business Administration, Vrije Universiteit, De Boelelaan 1105, 1081 HV Amsterdam, The Netherlands, \tt{l.stougie@vu.nl}.}

\author{Leo van Iersel, Steven Kelk, Nela Leki\'{c}, Leen Stougie\thanks{Leo van Iersel, Leen Stougie and Nela Leki\'{c} were respectively funded by Veni, CLS and Vrije Competitie grants from The Netherlands Organisation for Scientific Research (NWO). Leen Stougie was partially funded by the Tinbergen Institute.}}

\begin{document}

\maketitle

\begin{abstract}
Given two rooted phylogenetic trees on the same set of taxa $X$, the Maximum Agreement Forest problem~(\maf) asks to find a forest that is, in a certain sense, common to both trees and has a minimum number of components. The Maximum Acyclic Agreement Forest problem~(\maaf) has the additional restriction that the components of the forest cannot have conflicting ancestral relations in the input trees. There has been considerable interest in the special cases of these problems in which the input trees are required to be binary. However, in practice, phylogenetic trees are rarely binary, due to uncertainty about the precise order of speciation events. Here, we show that the general, nonbinary version of \maf has a polynomial-time 4-approximation and a fixed-parameter tractable (exact) algorithm that runs in $O(4^k \poly(n))$ time, where $n=|X|$ and~$k$ is the number of components of the agreement forest minus one. Moreover, we show that a $c$-approximation algorithm for nonbinary \maf and a $d$-approximation algorithm for the classical problem Directed Feedback Vertex Set (\dfvs) can be combined to yield a $d(c+3)$-approximation for nonbinary \maaf. The algorithms for \maf have been implemented and made publicly available.
\end{abstract}

\maketitle

\section{Introduction}

\emph{Background.} A rooted phylogenetic tree is a rooted tree with its leaves labelled by species, or strains of species, more abstractly
known as \emph{taxa}. Arcs, also called edges, are directed away from the root towards the leaves and internal vertices of the tree have indegree~1. It is a model used to exhibit ancestral relations between the species. For an
introduction to phylogenetic trees see \cite{MathEvPhyl,reconstructingevolution,SempleSteel2003}.

Occasionally it happens that on the same set of species topologically distinct phylogenetic trees are derived from different data sources (e.g. different genes). Such partly incompatible trees may arise due to biological reticulation events such as hybridization, recombination or lateral gene transfer \cite{HusonRuppScornavacca10,davidbook,Nakhleh2009ProbSolv}. These events cannot be explained in a tree-like ancestral relation model. There are additionally many non-reticulate biological phenomena, such as incomplete lineage sorting, that can likewise lead to conflicting tree signals  \cite{davidbook,Nakhleh2009ProbSolv}. Whatever the cause of the conflict, it is natural to wish to quantify the dissimilarity of two phylogenetic trees.

\emph{rSPR distance and Maximum Agreement Forests}. One
such measure of dissimilarity is the rooted Subtree Prune and Regraft (rSPR) distance, which asks for the minimum number of subtrees that need to be successively detached and re-attached in one of the trees to transform it into the other. The search for an alternative characterisation of rSPR distance was a major motivation behind the study of the agreement forest problem that
we now describe; a formal definition follows in the next section. We are given two rooted trees with the leaves labeled by the elements of a set~$X$ and no vertices with indegree and outdegree both equal to~1. An \emph{agreement
forest} is a partition of $X$ such that (a) in both trees the partition induces edge-disjoint subtrees and (b) for each block (``component'') of the partition, the two subtrees induced are 
phylogenetically compatible i.e. have a \emph{common
refinement}; see Figure~\ref{fig:maf}. The {\em maximum agreement forest problem} (\maf) is to find an agreement forest with a minimum number of components.

\clearpage

The binary \maf problem, in which the two input phylogenetic trees are binary, was introduced by Hein et al.~\cite{hein96}. Bordewich and Semple~\cite{BS05} showed that modulo a minor rooting technicality it is equivalent to computing the rSPR distance. The first correct polynomial-time approximation algorithm for this problem was the 5-approximation algorithm by Bonet et al.~\cite{bonet}. The first 3-approximation was given in~\cite{rsprFPT}, a quadratic-time 3-approximation in~\cite{3approxRSPR} and a linear-time 3-approximation in~\cite{whiddenFixed}. A fixed parameter tractable (FPT) algorithm for binary \maf was first given by Bordewich and Semple~\cite{BS05}, the running time of which was subsequently improved in~\cite{rsprFPT} to $O(4^kk^4+n^3)$ and in~\cite{whiddenFixed} to $O(2.42^kn)$. (See \cite{Flum2006,niedermeier2006} for an introduction to fixed parameter tractability). For more than two binary trees, there exists an 8-approximation~\cite{chataigner}.

However, in applied phylogenetics it rarely happens that trees are binary. Often there is insufficient information available to be able to determine the exact order in which several branching events occurred, and it is standard practice to model this uncertainty using vertices with outdegree higher than two: a \emph{soft polytomy} \cite{davidbook}. Hence it is extremely important to develop algorithms for the \emph{nonbinary} case i.e. when the trees are not necessarily binary and high-degree vertices capture a set of equally likely branching scenarios. The close relationship between rSPR distance and MAF also holds in the nonbinary case~\cite{josh}. However, compared to the binary case neither problem has been well-studied. Nonbinary agreement forests were introduced in~\cite{3approxRSPR}, where a $(d+1)$-approximation algorithm for nonbinary \maf was given, with~$d$ the maximum outdegree of the input trees. A kernelization was presented in~\cite{josh}, showing that nonbinary \maf is fixed-parameter tractable. Combining this kernelization of size~$64k$ with the trivial $O(n^k\text{poly}(n))$ time exact algorithm gives a $O((64k)^k\poly(n))$ time FPT algorithm, with~$n$ the number of species and~$k$ the size of the maximum agreement forest minus one. In addition, the kernel automatically implies the existence of a polynomial-time 64-approximation algorithm (since any instance trivially has a solution of size~$n$).

In this article we give a polynomial-time 4-approximation algorithm and an $O(4^k \poly(n))$ time FPT algorithm for nonbinary \maf. Both algorithms have been implemented and made publicly available. Although these results are of interest in their own right, we also show how to utilize them to obtain improved algorithms for computation of hybridization number.

\emph{Hybridization number and Maximum Acyclic Agreement Forests}. The other problem we study is a variation of \maf in which the roots of the subtrees in the agreement forest are not allowed to have conflicting ancestral relations in the two input phylogenetic trees. For an example, consider the agreement forest in Figure~\ref{fig:maf}. In the first phylogenetic tree, the subtree with leaves~$c$ and~$d$ is ``above'' the subtree with leaves~$a$ and~$b$, whereas it is the other way around in the second phylogenetic tree. By saying that a subtree is ``above'' another subtree, we mean that there exists a directed path from the root of the first to the root of the second subtree that contains at least one edge of the first subtree. Hence, the agreement forest in Figure~\ref{fig:maf} is, in this sense, cyclic. Such cyclic phenomena are prohibited in the {\em Maximum Acyclic Agreement Forest problem} \maaf, introduced in \cite{baroni05}. The main reason for studying MAAF, is its close connection with the \emph{hybridization number problem} (HN). A hybridization network (often also called a rooted phylogenetic network) is a rooted phylogenetic tree that additionally may contain \emph{reticulations}, vertices with indegree two or higher. Given two rooted rooted trees on the same set of taxa, the HN problem asks for a hybridization network with a minimum number of reticulations that displays (i.e. contains embeddings of) both the input trees. The HN problem can thus be viewed as constructing a ``most parsimonious'' explicit evolutionary hypothesis to explain the dissimilarity between the two input trees; the problem first gained attention following the publication of several seminal articles in 2004-5 e.g. \cite{baroni05,BaroniEtAl2004}. The number of reticulations in an optimal solution to the HN problem is exactly one less than the number of components in an optimal solution to the MAAF problem \cite{baroni05}, making the problems essentially equivalent.

The problem \maaf is NP-hard and APX-hard \cite{bordewich}, although there do exist efficient FPT algorithms for the binary variant of the problem e.g.
\cite{bordewich2,sempbordfpt2007,hybridnet,quantifyingreticulation,firststeps,whiddenFixed}. The
nonbinary variant of the problem has received comparatively little attention, although that too is FPT \cite{linzsemple2009,teresaFPT} and some algorithms have been implemented (see \cite{teresaFPT} for a discussion). In both cases, the practical applicability of the FPT algorithms is limited to instances of small or moderate size \cite{practicalCycleKiller}, for larger instances approximation algorithms are required. In that regard, it was recently shown in \cite{approximationHN} that, from an approximation perspective, binary MAAF is a very close relative of the classical problem {\em directed feedback vertex set} (DFVS), whose definition is presented in Section~\ref{sec:prelim}. Unfortunately, it is still a major open problem in approximation complexity whether DFVS permits a constant-factor polynomial-time approximation algorithm. The nonbinary variant of \maaf is of course at least as hard to approximate as the binary variant and is thus at least as hard to approximate as \dfvs. Hence, for large instances of \maaf, neither FPT algorithms nor polynomial-time approximation algorithms are, at the present time, appropriate. In \cite{practicalCycleKiller} we showed how, for the binary variant of MAAF, large instances \emph{can} however be very well approximated using a specific marriage of MAF and DFVS solvers. This approach is exponential-time in the worst case but in practice is very fast and yields highly competitive approximation factors. The question remained whether a similar approach would work for \emph{nonbinary} MAAF.
 
In this paper, we show that a $c$-approximation algorithm for nonbinary \maf and a $d$-approximation algorithm for \dfvs can be combined to yield a $d(c+3)$-approximation for nonbinary \maaf. Combining this with our aforementioned polynomial-time 4-approximation for \maf, we obtain a 7$d$-approximation for \maaf. As we discuss in the conclusion, it is likely that in practice $d=1$ can be obtained using ILP to solve the generated \dfvs instances. Hence, using the 4-approximation for nonbinary \maf we get an exponential-time 7-approximation for nonbinary \maaf and using the FPT algorithm for nonbinary \maf we get an exponential-time 4-approximation for nonbinary \maaf. If we wish a purely polynomial-time approximation, we can use the best-known polynomial-time approximation algorithm for \dfvs, yielding overall a polynomial-time $O(\log(k)\log(\log(k)))$-approximation for nonbinary \maaf, with~$k$ the size of a maximum acyclic agreement forest minus one. We mention here that our algorithms for \maf are partly based on the ideas in~\cite{whiddenFixed} and the algorithm for \maaf on the ideas in~\cite{practicalCycleKiller} but the analysis is different in both cases because nonbinary agreement forests are substantially different from binary agreement forests. In fact, high-outdegree vertices pose a formidable challenge because of the freedom to refine these vertices in one of exponentially many ways.

\begin{figure}
    \centering
    \includegraphics[scale=.8]{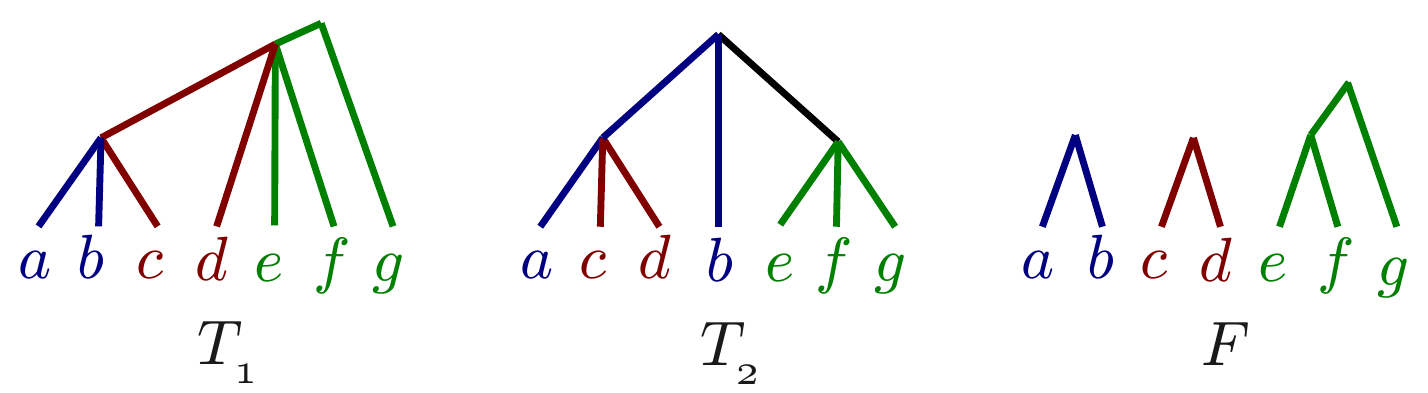}
    \caption{Two nonbinary rooted phylogenetic trees~$T_1$ and~$T_2$ and a maximum agreement forest~$F$ of~$T_1$ and~$T_2$.\label{fig:maf}}
\end{figure}

An implementation of both our \maf algorithms in Java has been made publicly available~\cite{mafprogram}.

\section{Preliminaries and statement of results}\label{sec:prelim}
Let~$X$ be a finite set (e.g. of species). A \emph{rooted phylogenetic} $\cX$-\emph{tree} is a rooted tree with no vertices with indegree-1 and outdegree-1, a root with indegree-0 and outdegree at least~2, and leaves bijectively labelled by the elements of~$X$. We identify each leaf with its label. We henceforth call a rooted phylogenetic $X$-tree a \emph{tree} for short. Note that we do not restrict phylogenetic trees to be binary. A tree~$T$ is a \emph{refinement} of a tree~$T'$ if~$T'$ can be obtained from~$T$ by contracting edges. For a tree~$T$ and a set $X'\subset X$, we define $T(X')$ as the minimal subtree of~$T$ that contains all elements of~$X'$, and $T|X'$ as the result of suppressing all vertices with in- and outdegree~1 of~$T(X')$. The set of leaves of a tree~$T$ is denoted~$L(T)$. We say that tree~$T'$ is displayed by tree~$T$ if~$T'$ can be obtained from a subtree of~$T$ by contracting edges. If~$T'$ is displayed by~$T$, then $T(L(T'))$ is the \emph{embedding} of~$T'$ in~$T$.

Throughout the paper, we usually refer to directed edges (arcs) simply as edges and to directed paths simply as paths. If~$e=(u,v)$ is an edge of some tree, then we say that~$v$ is a \emph{child} of~$u$, that~$u$ is the \emph{parent} of~$v$, that~$u$ is the \emph{tail} of~$e$, that~$v$ is the \emph{head} of~$e$ and we write $tail(e)=u$ and $head(e)=v$.

A \emph{forest} is defined as a set of trees. To avoid confusion, we call each element of a forest a \emph{component}, rather than a tree. Let~$T$ be a tree and~$\cF$ a forest. We say that~$\cF$ is a \emph{forest for}~$T$ if:
\begin{itemize}
\item each component~$F\in\cF$ is a refinement of $T|L(F)$;
\item the subtrees $\{T(L(F))\mid F\in \cF\}$ are edge-disjoint; and
\item the union of $L(F)$ over all $F\in \cF$ is equal to~$L(T)$.
\end{itemize}

By this definition, if~$\cF$ is a forest for some tree~$T$, then $\{L(F)\mid F\in\cF\}$ is a partition of the leaf set of~$T$. It will indeed sometimes be useful to see an agreement forest as a partition of the leaves, and sometimes to see it as a collection of trees.

If~$T_1$ and~$T_2$ are two trees, then a forest~$\cF$ is an \emph{agreement forest} of~$T_1$ and~$T_2$ if it is a forest for~$T_1$ and a forest for~$T_2$. The \emph{size} of a forest~$\cF$, denoted by~$|\cF|$, is defined as the number of its components. We consider the following computational problem. \\

\noindent{\bf Problem:} {\sc Nonbinary Maximum Agreement Forest} (Nonbinary \maf)\\
\noindent {\bf Instance:} Two rooted phylogenetic trees~$T_1$ and~$T_2$. \\
\noindent {\bf Solution:} An agreement forest~$\cF$ of~$T_1$ and~$T_2$.\\
\noindent {\bf Objective:} Minimize~$|\cF|-1$.\\

We define the \emph{inheritance graph} $IG(T_1,T_2,\cF)$ of an agreement forest, as the directed graph whose vertices are the components of~$\cF$ and which has an edge~$(F,F')$ precisely if either
\begin{itemize}
\item there exists a path in~$T_1$ from the root of $T_1(L(F))$ to the root of $T_1(L(F'))$ containing an edge of $T_1(L(F))$ or;
\item there exists a path in~$T_2$ from the root of $T_2(L(F))$ to the root of $T_2(L(F'))$ containing an edge of $T_2(L(F))$.
\end{itemize}

Note that if there exists a path in~$T_i$ (for $i\in\{1,2\}$) from the root of $T_i(L(F))$ to the root of $T_i(L(F'))$ containing an edge of $T_i(L(F))$, then this directly implies that this path also contains such an edge that has the root of $T_i(L(F))$ as tail.

An agreement forest~$\cF$ of~$T_1$ and~$T_2$ is called an \emph{acyclic agreement forest} if the graph $IG(T_1,T_2,\cF)$ is acyclic.

We call a forest an {\em $\cF$-splitting} if it is an acyclic agreement forest that can be obtained from a refinement of~$\cF$ by removing edges and suppressing vertices with in- and outdegree~1.

A \emph{maximum acyclic agreement forest} (\maaf) of~$T_1$ and~$T_2$ is an acyclic agreement forest of~$T_1$ and~$T_2$ with a minimum number of components. \\

\noindent{\bf Problem:} {\sc Nonbinary Maximum Acyclic Agreement Forest} (Nonbinary \maaf)\\
\noindent {\bf Instance:} Two rooted phylogenetic trees~$T_1$ and~$T_2$. \\
\noindent {\bf Solution:} An acyclic agreement forest~$\cF$ of~$T_1$ and~$T_2$.\\
\noindent {\bf Objective:} Minimize~$|\cF|-1$.\\

We use the notation MAF$(T_1,T_2)$ and MAAF$(T_1,T_2)$ for the optimal objective value of, respectively, a maximum agreement forest and a maximum acyclic agreement forest for input trees~$T_1$ and~$T_2$. Hence, if~$\cA$ is a maximum agreement forest and~$\cM$ is a maximum acyclic agreement forest, then MAF$(T_1,T_2)=|\cA|-1$ and MAAF$(T_1,T_2)=|\cM|-1$.

We note that in the MAF and MAAF literature it is commonplace to assume that there is a leaf labelled $\rho$ attached to the root of each input tree: see \cite{approximationHN} and earlier articles. In this article we omit $\rho$. The extra leaf $\rho$ was originally introduced into the binary MAF literature to ensure that optimum solutions to MAF correctly correspond to optimum solutions to the rSPR problem~\cite{BS05}. If one defines MAF, as we do, without $\rho$, then we can easily simulate the involvement of $\rho$ by introducing a new taxon $x' \not \in X$ which is a child of the root in both the input trees. The MAAF literature grew out of the MAF literature and thus inherited the use of $\rho$. However, for MAAF $\rho$ is not necessary at all, and this is tightly linked to the acyclic character of the inheritance graph. Although we omit the proof, it is easy to show that if we have a solution to MAAF in which $\rho$ appears on its own in an isolated component, we can obtain a better solution by grafting $\rho$ onto any of the components $C$ that have indegree~0 in the inheritance graph. This can create new outgoing edges from $C$ in the inheritance graph, but it cannot create new incoming edges, and hence preserves the acyclicity of the graph. For both these reasons (the ease with which $\rho$ can be simulated in the case of MAF, and its redundancy in the case of MAAF) we choose to omit $\rho$ from this paper.

The \maaf problem is closely related to an important problem in phylogenetics. A \emph{rooted phylogenetic network} is a directed acyclic graph with no vertices with indegree~1 and outdegree~1 and leaves bijectively labelled by the elements of~$X$. Rooted phylogenetic networks, will henceforth be called \emph{networks} for short in this paper.  A tree~$T$ is \emph{displayed} by a network~$N$ if~$T$ can be obtained from a subtree of~$N$ by contracting edges. Using~$\delta^-(v)$ to denote the indegree of a vertex~$v$, a \emph{reticulation} is a vertex~$v$ with~$\delta^-(v)\geq 2$. The \emph{reticulation number} of a network~$N$ is given by

\[
r(N)=\sum_{v: \delta^-(v)\geq 2}(\delta^-(v)-1).
\]

In \cite{baroni05} it was shown that, in the binary case, the optimum to {\sc MAAF} is equal to the optimum of the following problem. Later, in \cite{linzsemple2009}, this characterisation
was extended to nonbinary trees.

\noindent{\bf Problem:} {\sc Nonbinary Minimum Hybridization} (Nonbinary MH)\\
\noindent {\bf Instance:} Two rooted phylogenetic trees $T_1$ and $T_2$. \\
\noindent {\bf Solution:} A phylogenetic network~$N$ that displays $T_1$ and $T_2$.\\
\noindent {\bf Objective:} Minimize~$r(N)$.\\

Moreover, it was shown that, for two trees~$T_1,T_2$, \emph{any} acyclic agreement forest for~$T_1$ and~$T_2$ with~$k+1$ components can be turned into a phylogenetic network that displays~$T_1$ and~$T_2$ and has reticulation number~$k$, and vice versa. Thus, any approximation for Nonbinary \maaf gives an approximation for nonbinary {\sc MH}.

There is also a less obvious relation between \maaf and the \emph{directed feedback vertex set problem}, a problem which is well-known in the communities of theoretical computer science and combinatorial optimisation. A feedback vertex set of a directed graph is a subset of the vertices that contains at least one vertex of each directed cycle. Equivalently, a subset of the vertices of a directed graph is a \emph{feedback vertex set} if removing these vertices from the graph makes it acyclic.\\

\noindent{\bf Problem:} Directed Feedback Vertex Set (\dfvs)\\
\noindent {\bf Instance:} A directed graph~$D$. \\
\noindent {\bf Solution:} A feedback vertex set~$S$ of $D$. \\
\noindent {\bf Objective:} Minimize~$|S|$.\\

In the next two sections we will prove the following theorems. In Section~3, we relate approximability of Nonbinary \maaf to that of Nonbinary \maf and \dfvs.

\begin{theorem}
\label{thm:maafapprox}
If there exists a $c$-approximation for Nonbinary \maf and a $d$-approximation for \dfvs, then there exists a $d(c+3)$-approximation for Nonbinary \maaf and hence for Nonbinary {\sc MH}.
\end{theorem}

In Section~4, we design a polynomial-time approximation algorithm for (Nonbinary) \maf and prove constant factor approximability.

\begin{theorem}\label{thm:maf}
There is a polynomial-time 4-approximation for (nonbinary) \maf.
\end{theorem}

Moreover, the proof of Theorem~\ref{thm:maf} almost directly leads to fixed-parameter tractability.

\begin{theorem}\label{thm:fpt}
Nonbinary \maf can be solved exactly in $O(4^k \poly(n))$ time, with~$n$ the number of leaves and~$k$ the number of components of a maximum agreement forest minus one.
\end{theorem}

Combining Theorems~\ref{thm:maafapprox} and~\ref{thm:maf} above, we obtain the following corollary.

\begin{corollary}
If there exists a $d$-approximation for \dfvs then there exists a $7d$-approximation for Nonbinary \maaf and hence for Nonbinary {\sc MH}.
\end{corollary}

Moreover, using the $O(\log(\tau)\log(\log(\tau)))$-approximation for weighted \dfvs from~\cite{dfvsApprox}, with~$\tau$ the weight of a minimum feedback vertex set, we also obtain the following.

\begin{corollary}
There exists a polynomial-time $O(\log(k)\log(\log(k)))$-approximation for nonbinary \maaf, with~$k$ the number of components of a maximum acyclic agreement forest minus one.
\end{corollary}

\section{Approximating nonbinary MAAF}

In this section we prove Theorem~\ref{thm:maafapprox}. We will show that, for two nonbinary trees, \maaf can be approximated by combining algorithms for \maf and \dfvs.

An agreement forest~$\cA$ is said to be \emph{maximal} if there is no agreement forest that can be obtained from~$\cA$ by merging components. It is clear that, given any agreement forest, a maximal agreement forest with at most as many components can be obtained in polynomial time.

Let~$T_1$ and~$T_2$ be two nonbinary trees. Consider a maximal agreement forest~$\cA$ and a maximum acyclic agreement forest~$\cM$ for these two trees. We will first prove that there exists an $\cA$-splitting of size at most $|\cA|+3|\cM|$. After that, we will show how the problem of finding an optimal $\cA$-splitting can be reduced to \dfvs.

The idea of the first part of the proof is to split components of~$\cA$ according to~$\cM$. We show that to make~$\cA$ acyclic we will increase the number of components of~$\cA$ by at most three times the size of~$\cM$.

We start with some definitions. In the first part of the proof, we see an agreement forest~$\cA$ for~$T_1$ and~$T_2$ as a partition of the leaf set~$X$ for which holds that:
\begin{enumerate}
\item $T_1|A$ and~$T_2|A$ have a common refinement, for all~$A\in\cA$; and
\item the subtrees $\{T_i(A)\mid A\in\cA\}$ are edge-disjoint, for~$i\in\{1,2\}$.
\end{enumerate}

Using this definition of agreement forests, an $\cA$-splitting is an acyclic agreement forest that can be obtained by splitting components of~$\cA$. The following observation is easily verifiable.

\begin{observation}\label{obs:acylic}
If~$\cM$ is an acyclic agreement forest and~$\cM'$ is an agreement forest that can be obtained from~$\cM$ by splitting components, then~$\cM'$ is an acyclic agreement forest.
\end{observation}

For a component~$A$ of an agreement forest for two trees~$T_1$ and~$T_2$, we write $r_i(A)$ to denote the root of $T_i(A)$, for~$i\in\{1,2\}$. For two components~$M$ and~$M'$ of~$\cM$, we say that~$M$ is \emph{lower} than~$M'$ if in the inheritance graph of~$\cM$ there is a directed path (and hence an edge) from~$M'$ to~$M$. Since the inheritance graph of~$\cM$ is acyclic,~$\cM$ contains some lowest element. Moreover, for any subset of components of~$\cM$, there exists a component that is lowest over that subset. For a component~$A$ of~$\cA$ and a component~$M$ of~$\cM$, we say that~$M$ {\em properly intersects}~$A$ (and that~$A$ is~\emph{properly intersected by}~$M$) if $M\cap A \neq \emptyset$ and $A\setminus M\neq \emptyset$.

We are now ready to describe the procedure for splitting~$\cA$. Initially, all components of~$\cA$ and~$\cM$ are unmarked. We iteratively choose a component~$M^*$ of~$\cM$ that is lowest over all unmarked components of~$\cM$. For each component~$A$ of~$\cA$ that is properly intersected by~$M^*$, we split~$A$ into two components $A\cap M^*$ and $A\setminus M^*$ and we mark $A\cap M^*$. Then we mark~$M^*$ and any unmarked components~$A'$ of~$\cA$ with $A'\subseteq M^*$ and proceed to the next iteration. We continue this procedure until all components of~$\cM$ are marked.

It is clear that, if~$\cA^*$ is the agreement forest obtained at the end of some iteration, then no marked component of~$\cM$ properly intersects any component of~$\cA^*$. Moreover, no marked component of~$\cA^*$ is properly intersected by any component of~$\cM$.

\begin{lemma}
\label{lem:edgedisjoint}
Let~$\cA$ be an agreement forest for~$T_1$ and~$T_2$, let~$M^*$ be a component of~$\cM$ that is lowest over all unmarked components of~$\cM$ and let~$A$ be a component of~$\cA$ that is properly intersected by~$M^*$. Then, $T_i(A \cap M^*)$ and $T_i(A\setminus M^*)$ are edge-disjoint subtrees of~$T_i$, for $i\in\{1,2\}$.
\end{lemma}
\begin{proof}
Suppose to the contrary that in at least one tree, say~$T_1$, there exists an edge~$e$ such that $e\in T_1(A \cap M^*)$ and $e\in T_1(A\setminus M^*)$. Consider the set of leaves~$S$ that can be reached from~$e$. Clearly, some leaf of $A\setminus M^*$ has to be in~$S$. Let $x\in S \cap (A\setminus M^*)$. Clearly,~$x$ must be in some component of~$\cM$ other than~$M^*$. Call that component~$M'$. Since components of~$\cM$ are edge disjoint, $r_1(M')$ has to be below~$e$. Because $e\in T_1(A \cap M^*)$, it follows that $e\in T_1(M^*)$ and hence that~$M'$ is lower than~$M^*$. Since $M^*$ is lowest over all unmarked components of~$\cM$, it follows that~$M'$ is marked and hence that~$M'$ does not properly intersect any component of~$\cA$. This is a contradiction because~$M'$ properly intersects~$A$.
\end{proof}

Lemma \ref{lem:edgedisjoint} shows that when we split a component, the resulting two subtrees are edge-disjoint. This property holds for all components properly intersected by $M^{*}$. Observe that when we split a component, the two newly-created subtrees cannot possibly share edges with other components, because of the assumption that at the start of the iteration all the subtrees were edge-disjoint. Hence, at the end of the iteration, all of the subtrees are mutually edge-disjoint. To show that we still have an agreement forest, it is necessary to show that the components still obey the refinement criterion. This follows from the following (unsurprising) observation.

\begin{observation}
\label{obs:stillrefine}
Let $T^{*}_1$ and $T^{*}_2$ be two trees on taxon set $X^{*}$ such that $T^{*}_1$ and $T^{*}_2$ have a common refinement. Then, for any $X^{\dagger} \subseteq X^{*}$,  $T^{*}_1 | X^{\dagger}$ and $T^{*}_2 | X^{\dagger}$ have a common refinement.
\end{observation}

We have now shown that the result of each iteration is an agreement forest for~$T_1$ and~$T_2$ such that no marked component of~$\cM$ properly intersects any component of this agreement forest. Let~$\cA'$ be the agreement forest obtained at the end of the last iteration. Since at the end of the procedure all components of~$\cM$ are marked, no component of~$\cM$ properly intersects any component of~$\cA'$. It follows that~$\cA'$ can be obtained from~$\cM$ by splitting components. Since~$\cM$ is acyclic, it follows by Observation~\ref{obs:acylic} that~$\cA'$ is acyclic. Hence, $\cA'$ is an $\cA$-splitting. It remains to bound the size of this $\cA$-splitting. To do so, we will need the following observation and lemma.

\begin{observation}\label{obs:merge}
Let~$A_1$ and~$A_2$ be components of some agreement forest for~$T_1$ and~$T_2$. If~$A_1$ and~$A_2$ have the same root in both~$T_1$ and~$T_2$, then the result of merging~$A_1$ and~$A_2$ into a single component $A_1\cup A_2$ is still an agreement forest of~$T_1$ and~$T_2$.
\end{observation}

\begin{lemma}
\label{lem:maximal} If~$\cA$ is the agreement forest at the beginning of some iteration, then there are no four unmarked components of~$\cA$ that have a common vertex in both trees.
\end{lemma}
\begin{proof}
First let~$\cA$ be the agreement forest at the beginning of the first iteration. Suppose that $A_1,A_2,A_3,A_4$ are unmarked components of~$\cA$ and that~$v_i$ is common to $T_i(A_1),\ldots ,T_i(A_4)$ for~$i\in\{1,2\}$. For each~$i\in\{1,2\}$, there is at most one~$j\in\{1,2,3,4\}$ for which~$T_i(A_j)$ contains the edge entering~$v_i$. Hence, there are at least two components, say~$A_1$ and~$A_2$, that do not use this edge in either tree. It follows that~$r_i(A_1)=r_i(A_2)=v_i$ for~$i\in\{1,2\}$. However, then~$A_1$ and~$A_2$ can be merged into a single component by Observation~\ref{obs:merge}, which is a contradiction because~$\cA$ is maximal.

We have shown that the lemma is true at the beginning of the first iteration. Now assume that it is true at the beginning of some iteration. Each component that is split during the iteration is split into one marked and one unmarked component. Hence, for each vertex, the number of unmarked components using that vertex does not increase. It follows that the lemma is still true at the end of the iteration. This completes the proof.
\end{proof}

\begin{lemma}
\label{lem:touching_afs}
Let~$\cA$ be the agreement forest at the beginning of some iteration. If~$M^*$ is a component of~$\cM$ that is lowest over all unmarked components of~$\cM$, then~$M^*$ properly intersects at most three components of~$\cA$.
\end{lemma}
\begin{proof}
Let~$A$ be a component of~$\cA$ that is properly intersected by~$M^*$. We claim that there is a directed path from~$r_i(A)$ to~$r_i(M^*)$ for~$i\in\{1,2\}$ (possibly, $r_i(A)=r_i(M^*)$). To see this, first note that, since $A\cap M^*\neq\emptyset$, there must be either a directed path from~$r_i(A)$ to~$r_i(M^*)$, or from~$r_i(M^*)$ to~$r_i(A)$, for~$i\in\{1,2\}$. Suppose that this path goes from~$r_i(M^*)$ to~$r_i(A)$ and contains at least one edge. Since $A\cap M^*\neq\emptyset$, this path has to contain at least one edge of~$T_i(M^*)$. Now observe that, since~$M^*$ properly intersects~$A$, there exists some $a\in A\setminus M^*$, which is in some component of~$\cM$, say in~$M'$. Note that~$M'$ is unmarked since it properly intersects~$A$. However, then we obtain a contradiction because~$M'$ is lower than~$M^*$, while~$M^*$ is lowest over all unmarked components of~$\cM$. Hence, there is a directed path from~$r_i(A)$ to~$r_i(M^*)$ for~$i\in\{1,2\}$. This path is contained in~$T_i(A)$ because $A\cap M^*\neq\emptyset$.

Now assume that the lemma is not true, i.e. that there exist four components $A_1,A_2,A_3,A_4$ of~$\cA$ that are properly intersected by~$M^*$. We have seen that there is a directed path from $r_i(A_j)$ to~$r_i(M^*)$, which is contained in~$T_i(A_j)$, for~$i\in\{1,2\}$ and~$j\in\{1,2,3,4\}$. Hence, $r_i(M^*)$ is common to all four components~$A_1,A_2,A_3,A_4$. Moreover, $A_1,A_2,A_3,A_4$ are all unmarked since they are properly intersected by~$M^*$. This is a contradiction to Lemma~\ref{lem:maximal}.
\end{proof}

Lemma~\ref{lem:touching_afs} shows that each iteration splits at most three components. Each split adds one component. Thus, for every component~$M$ of~$\cM$, the number of components of the $\cA$-splitting is increased by at most three, which concludes the proof of the following theorem.

\begin{theorem}
\label{thm:3maaf} Let~$T_1$ and~$T_2$ be two (nonbinary) trees. If~$\cA$ is a maximal agreement forest and~$\cM$ a maximum acyclic agreement forest of~$T_1$ and~$T_2$, then there exists an $\cA$-splitting of size at most $|\cA|+3|\cM|$.
\end{theorem}

Now, suppose we have a $c$-approximation~$\cA$ for \maf; i.e.,
\begin{displaymath}
  |\cA|-1 \leq c \cdot \maf(T_1,T_2) \leq c \cdot \maaf(T_1,T_2).
\end{displaymath}
Let $OptSplit(\cA)$ denote the size of an $\cA$-splitting with smallest number of components. The last inequality together with Theorem~\ref{thm:3maaf} imply that
\begin{displaymath}
  OptSplit(\cA)-1 \leq |\cA|-1 + 3 \cdot \maaf(T_1, T_2) \leq (c+3)\cdot \maaf(T_1, T_2).
\end{displaymath}

The remaining part of the proof will be accomplished by reducing the problem of finding an optimal $\cA$-splitting to \dfvs in such a way that a $d$-approximation algorithm for \dfvs gives a $d$-approximation for an optimal $\cA$-splitting. Combining this with the above inequality, the proof of Theorem~\ref{thm:maafapprox} will follow.

\begin{figure}
 \centering
 \includegraphics[scale=.8]{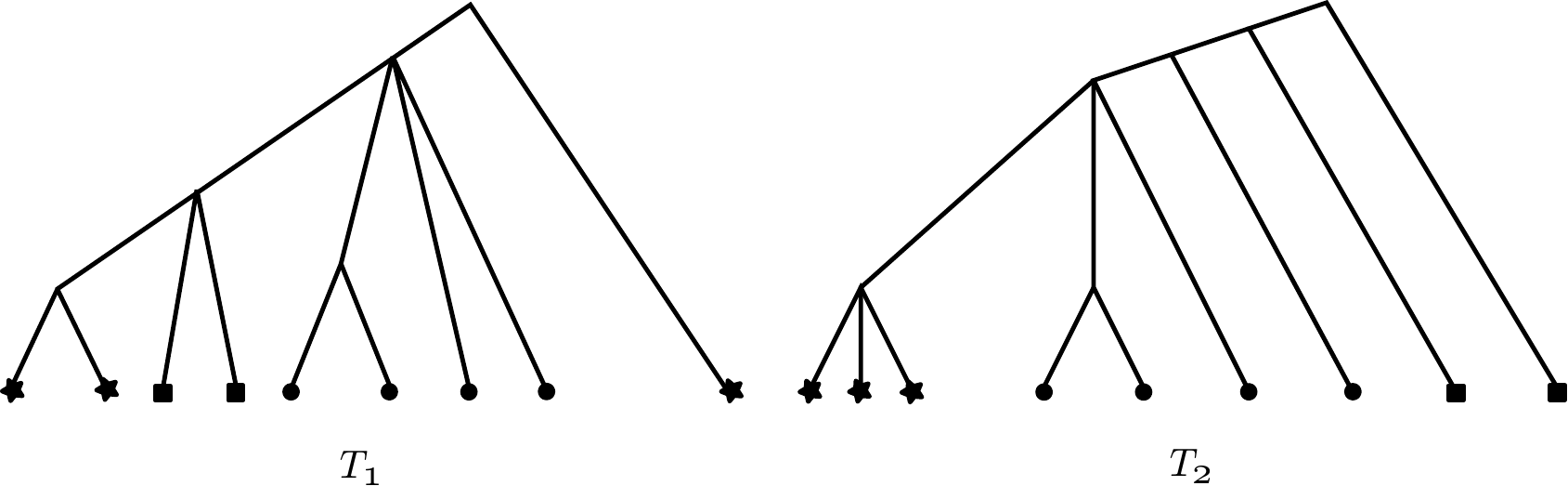}
 \vspace{.5cm}
 \includegraphics[scale=.9]{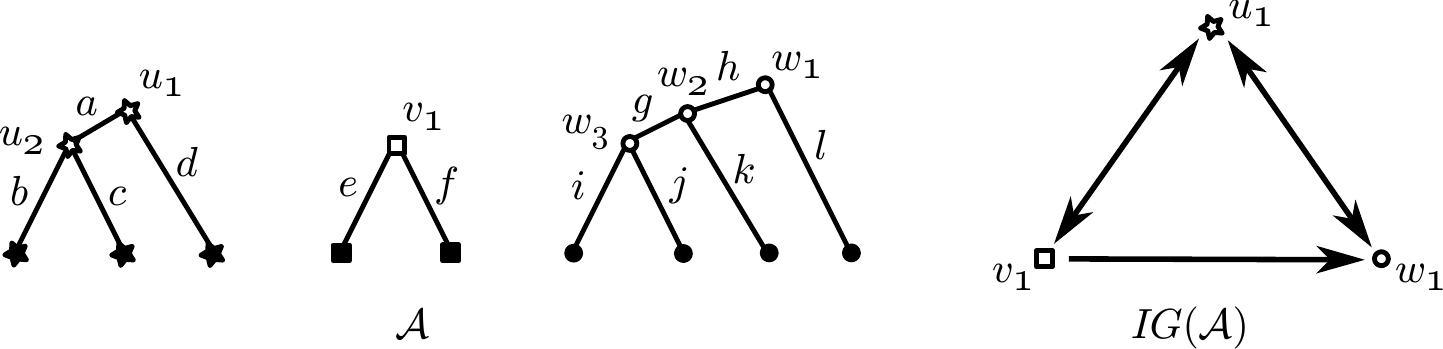}
 \caption{Two trees~$T_1$ and~$T_2$, an agreement forest~$\cA$ for~$T_1$ and~$T_2$ and its inheritance graph $IG(\cA)$. Leaf labels have been omitted.\label{fig:reduc1}}
\end{figure}

\begin{figure}
 \centering
 \includegraphics[scale=.8]{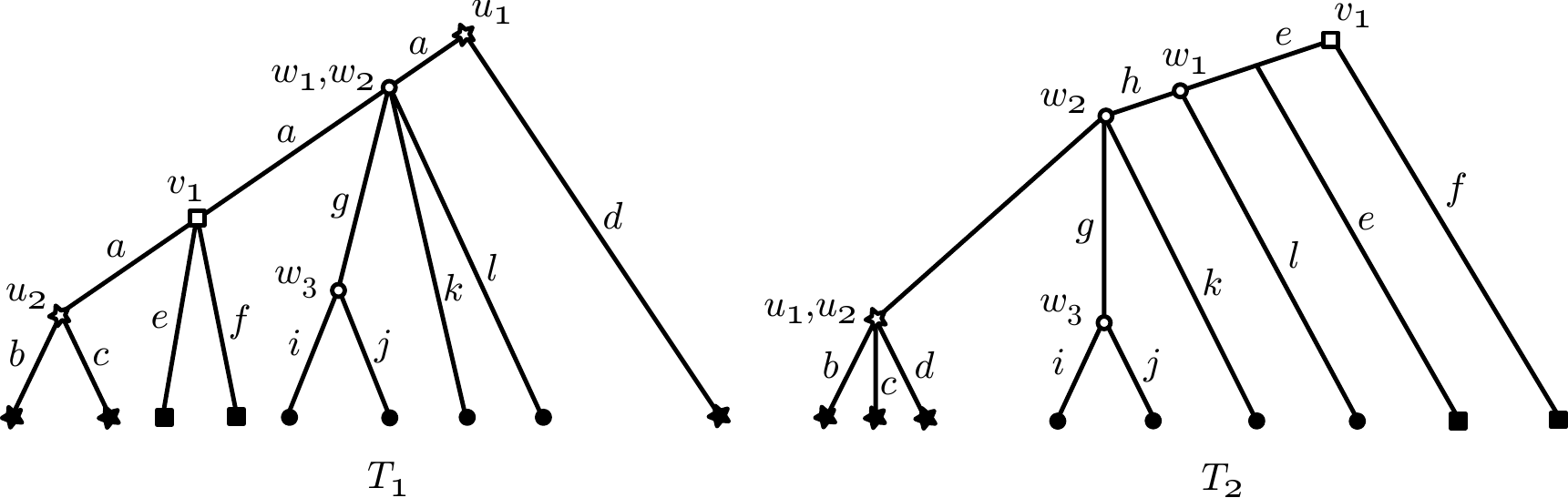}
 \caption{The trees~$T_1$ and~$T_2$ from Figure~\ref{fig:reduc1} labelled with the internal vertices and all edges of~$\cA$.\label{fig:reduc3}}
\end{figure}

\begin{figure}
 \centering
 \includegraphics[scale=.9]{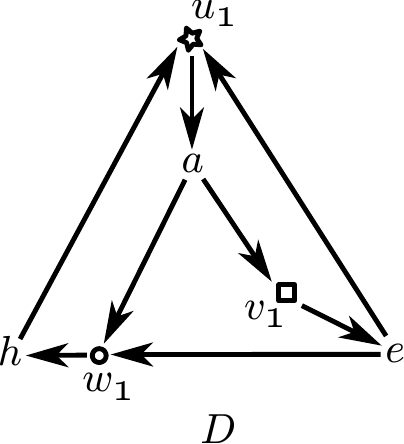}
 \caption{Part of the input graph~$D$ for \dfvs for the trees~$T_1,T_2$ and agreement forest~$\cA$ of Figures~\ref{fig:reduc1} and~\ref{fig:reduc3}. Vertices that do not appear in any directed cycle of~$D$ have been omitted. Minimum feedback vertex sets of~$D$ are $\{u_1\}$ and $\{a\}$.\label{fig:reduc4}}
\end{figure}

From now on, we see an agreement forest as a collection of trees, as specified in the preliminaries section. We will assume that the components of~$\cA$ are never more refined than necessary (i.e. there is no agreement forest of~$T_1$ and~$T_2$ that can be obtained from~$\cA$ by contracting edges).

To prepare for the construction of an input graph to \dfvs, we label the vertices and edges of~$T_1$ and~$T_2$ by the vertices and edges of~$\cA$ that they correspond to. Note that each vertex of~$\cA$ is used exactly once as a label but, due to refinements, some vertices of the trees can have multiple labels. Moreover, some edges of~$\cA$ can be used as a label multiple times (when the edge corresponds to a path in a tree) but each edge is used as a label at least once in at least one tree (by the assumption that~$\cA$ is not more refined than necessary). For example, for the trees and agreement forest in Figure~\ref{fig:reduc1}, the labelling is illustrated in Figure~\ref{fig:reduc3}.

We construct an input graph~$D$ for \dfvs as follows. For every internal vertex of $\cA$, we create a vertex for~$D$. Denote the set of these vertices by~$V_V(D)$. In addition, for every edge of~$\cA$ we create a vertex for~$D$. This set of vertices is denoted by~$V_E(D)$. We write $V(D)= V_V(D) \cup V_E(D)$. We create edges of~$D$ as follows. For every $v \in V_V(D)$ and every $e\in V_E(D)$ create an edge $(v,e)$ if $tail(e)=v$ in~$\cA$ and we create an edge $(e,v)$ if~$\hat{v}$ can be reached from $head(\hat{e})$ in at least one tree, for some edge~$\hat{e}$ labelled by~$e$ and the vertex~$\hat{v}$ labelled by~$v$. See Figure~\ref{fig:reduc4} for an example.

We will show that any feedback vertex set of~$D$ corresponds to an $\cA$-splitting and vice-versa. Moreover, after appropriately weighting the vertices of~$D$, the number of ``splits'' of the $\cA$-splitting (i.e. the size of the $\cA$-splitting minus the size of~$\cA$) will be equal to the weight of the corresponding weighted feedback vertex set.

Let~$F\subset V(D)$. In what follows, we use~$F$ both for sets of vertices in~$D$ and for the set of vertices and edges they represent in~$\cA$. Let $\cA \setminus F$ be the forest obtained from~$\cA$ by removing the vertices and edges of~$F$ and repeatedly removing indegree-0 outdegree-1 vertices and suppressing indegree-1 outdegree-1 vertices.

\begin{lemma}
A subset~$F$ of~$V(D)$ is a feedback vertex set of~$D$ if and only if $\cA \setminus F$ is an $\cA$-splitting.
\end{lemma}
\begin{proof}
It is clear that $\cA \setminus F$ is an agreement forest of~$T_1$ and~$T_2$. Hence, it is enough to show that $D\setminus F$ has a directed cycle if and only if $IG(\cA \setminus F)$ has a directed cycle.

First, let $C_1, C_2,\ldots, C_k=C_1$ be a cycle in $IG(\cA \setminus F)$. We will show that this implies that there exists a cycle in $D\setminus F$. Let $u_1, u_2,\ldots, u_k$ be the roots of the components $C_1, C_2,\ldots ,C_k$, respectively. Notice that $u_1, u_2,\ldots ,u_k$ are internal vertices of~$\cA$ and hence represented in $D$. Moreover, since these vertices are in $\cA \setminus F$, they are also in $D\setminus F$.

We prove this side of the lemma by showing that the presence of an edge $(C_i, C_{i+1})$ of $IG(\cA \setminus F)$ implies the existence of a directed path from~$u_i$ to~$u_{i+1}$ in $D\setminus F$, for $i=1,\ldots ,k-1$. By the definition of the inheritance graph, an edge $(C_i, C_{i+1})$ of $IG(\cA \setminus F)$ implies that there exists a directed path from~$u_i$ to~$u_{i+1}$ that uses an edge of~$C_i$ in at least one of the trees. It is easy to see that this directed path uses an edge,~$a$ say, of~$C_i$ that has~$u_i$ as its tail. Moreover, since~$C_i$ and~$C_{i+1}$ are components of~$\cA\setminus F$, $u_1,a$ and~$u_{i+1}$ are vertices of $D \setminus F$ and $(u_i,a)$ and $(a,u_{i+1})$ are edges of $D\setminus F$, forming a directed path from~$u_i$ to~$u_{i+1}$ in $D\setminus F$.

Now, assume that $D\setminus F$ contains a cycle. Let $u_1,...,u_k=u_1$ be a longest cycle. We will show that this implies the existence of a cycle in $IG(\cA \setminus F)$. Clearly, $u_1,\ldots ,u_k$ cannot all correspond to vertices and edges from the same component of~$\cA\setminus F$. From the assumption that $u_1,...,u_k=u_1$ is a longest cycle, it can be argued that there are at least two vertices among $u_1,\ldots ,u_k$ that are roots of components of~$\cA\setminus F$. Let $r_1,\ldots ,r_s=r_1$ denote those vertices of the cycle that correspond to roots of components. Let $C_1,\ldots ,C_s$ be the corresponding components of $\cA\setminus F$, respectively. We will show that $C_1,\ldots ,C_s$ form a cycle in $IG(\cA \setminus F)$.

For $i=1,\ldots ,s-1$, we show that there exists an edge $(C_i, C_{i+1})$ in $IG(\cA \setminus F)$. First observe that there exists a directed path from~$r_i$ to~$r_{i+1}$ in $D\setminus F$ and that, since~$D$ is bipartite, this path is of the form
\[
(r_i=v_0, e_1, v_1, e_2, \ldots ,e_t,v_t=r_{i+1})
\]
where $e_1,\ldots ,e_t$ are edges of $\cA \setminus F$ and $v_1,\ldots ,v_{t-1}$ are internal vertices of~$C_i$.

First assume $t=1$. In this case, the path is just $(r_i, e_1,r_{i+1})$. By the definition of~$D$, $tail(e_1)=r_i$ in~$C_i$ and there is a directed path from $head(\hat{e}_1)$ to~$r_{i+1}$, with~$\hat{e}_1$ some edge labelled by~$e_1$, in at least one of the two trees (such an edge $\hat{e}_1$ definitely exists because of the assumption that~$\cA$ is not more refined that necessary). This tree then contains a path from~$r_i$ to~$r_{i+1}$ that contains edge~$\hat{e}_1$ of~$C_i$. Hence there is an edge $(C_i, C_{i+1})$ in $IG(\cA \setminus F)$.

Now assume $t>1$. For~$j=1,\ldots ,t$, by the definition of~$D$, $tail(e_j)=v_{j-1}$ in~$C_i$ and there is a directed path from $head(\hat{e}_j)$ to~$v_j$, with~$\hat{e}_j$ some edge labelled by~$e_j$, in at least one of the two trees. For $j<t$, $e_j$ and~$v_j$ are both in the same component~$C_i$, implying that there is a directed path from $head(e_j)$ to~$v_j$ in~$C_i$. Hence, there is a directed path from~$r_i$ to~$v_{t-1}$ in both trees. Thus, the tree that contains a directed path from~$v_{t-1}$ to~$r_{i+1}$ also contains a directed path from~$r_i$ to~$r_{i+1}$ containing edge~$\hat{e}_t$ of~$C_i$. Hence, $(C_i, C_{i+1})$ is an edge of $IG(\cA \setminus F)$.
\end{proof}

The above lemma showed that we can turn~$\cA$ into an $\cA$-splitting by removing vertices and edges corresponding to a feedback vertex set of~$D$. We will now add weights to the vertices of~$D$ in order to enforce that an optimal feedback vertex set gives an optimal $\cA$-splitting and, moreover, that an approximate feedback vertex set gives an approximate $\cA$-splitting.

We define a weight function~$w$ on vertices of~$D$ as follows, using $\delta^+_\cA(v)$ to denote the outdegree of a vertex~$v$ in~$\cA$.
\begin{displaymath}
  w(v) =
  \begin{cases}
    \delta^+_\cA(v)-1 &\mbox{if } v \in V_V(D)\\
    1 				&\mbox{if } v \in V_E(D).
  \end{cases}
\end{displaymath}

The weight of a feedback vertex set~$F$ is defined as $w(F)= \sum_{v \in F} w(v)$.

Intuitively, the weight of each vertex~$v$ equals the number of components its removal would add to the $\cA$-splitting.

To make this precise, we need the following definition. We call a feedback vertex set~$F$ \emph{proper} if it is minimal (i.e. no proper subset of~$F$ is a feedback vertex set) and if for every vertex~$v\in V_V(D)$ at most $\delta^+_D(v)-2$ children of~$v$ in~$D$ are contained in~$F$, with $\delta^+_D(v)$ denoting the outdegree of~$v$ in~$D$. (Recall that the children of~$v$ in~$D$ are elements of~$V_E(D)$). For example, proper feedback vertex sets of the graph~$D$ in Figure~\ref{fig:reduc4} are $\{u_1\}$ and $\{v_1,w_1\}$. The idea behind this definition is that when all, or all but one, of the children of~$v$ are in~$F$, then we could just as well add~$v$ instead, which does not increase the total weight of the feedback vertex set. Hence, given any feedback vertex set~$F$, a proper feedback vertex set~$F'$ with~$w(F')\leq w(F)$ can be found in polynomial time.

\begin{lemma}
If~$F$ is a proper feedback vertex set of~$D$, then $\cA\setminus F$ is an $\cA$-splitting of size $|\cA| + w(F)$.
\end{lemma}
\begin{proof}
For an edge~$e$ of~$\cA$, we use $tail(e)$ and $head(e)$ to refer to the tail and head of~$e$ in~$\cA$, respectively.

Consider a vertex of~$F$ corresponding to an edge~$e$ of~$\cA$. Then, every cycle in~$D$ that contains~$e$ also contains $tail(e)$. Hence, since~$F$ is minimal, $tail(e)$ is not contained in~$F$. Moreover, suppose that $tail(e)$ is not the root of a component of~$\cA$ and let~$e'$ be the edge with $head(e')=tail(e)$. Then, for every cycle in~$D$ containing~$e$ but not~$e'$ (and hence containing $tail(e)$ but not $tail(e')$), replacing~$e$ and $tail(e)$ by~$e'$ and $tail(e')$ gives again a directed cycle in~$D$. Hence, since~$F$ is minimal, it follows that either~$e'$ or~$tail(e')$ is contained in~$F$, but not both.

Now consider a vertex of~$F$ corresponding to a vertex~$v$ of a component~$C$ of~$\cA$. Then, by the previous paragraph, no edge~$e$ with $tail(e)=v$ is contained in~$F$. Moreover, if~$v$ is not the root of~$C$ and~$e'$ is the edge with $head(e')=v$, then, as before, either~$e'$ or~$tail(e')$ is contained in~$F$, but not both.

To summarize the previous two paragraphs, if~$F$ contains a vertex corresponding to a vertex~$v$ of~$\cA$ that is \emph{not} the root of its component, then~$F$ also contains either the edge entering~$v$ or the parent of~$v$. Similarly, if~$F$ contains a vertex corresponding to an edge~$e$ of~$\cA$ that is \emph{not} leaving the root of its component, then~$F$ also contains either the edge entering~$tail(e)$ or the parent of~$tail(e)$. Informally speaking, this means that if you remove something from a component, you also remove everything above it.

More formally, it follows that $\cA\setminus F$ can be obtained from~$\cA$ by, repeatedly, either removing the root of a component or removing a subset of the edges leaving the root of a component. Moreover, if we remove a subset of the edges leaving the root of a component, at least two of these edges are not removed, because~$F$ is proper.

Removing edges and vertices from~$\cA$ in this way, it is easy to see that the number of components is increased by~$w(F)$.
\end{proof}

To complete the proof, let~$\cA$ be a $c$-approximation to \maf, $F$ a minimal feedback vertex set that is a $d$-approximation to weighted \dfvs on~$D$ and~$F^*$ an optimal  solution to weighted \dfvs on~$D$, then we have:
\begin{align*}
 |\cA\setminus F| - 1 & = |\cA| + w(F) - 1\\
				&\leq |\cA| + d\cdot w(F^*) - 1\\
  				&\leq d(|\cA| + w(F^*) - 1)\\
 				&= d(OptSplit(\cA)-1) \\
 				&\leq d(c+3)MAAF(T_1,T_2).
\end{align*}

We have thus shown how to construct an agreement forest that is a $d(c+3)$-approximation to \maaf and with that we conclude the proof of Theorem~\ref{thm:maafapprox}.

\section{Nonbinary MAF}

In the first part of this section we present a 4-approximation algorithm for \maf, proving Theorem \ref{thm:maf}. The performance analysis leads almost straightforwardly to a fixed parameter tractability result, which we present in the second part.

\subsection{Approximation of nonbinary MAF}

Let~$T_1$ and~$T_2$ be the input trees to \maf. They do not have to be binary. We will construct a forest by ``cutting''~$T_2$. A ``cut'' operation of $T_2$ consists of removing an edge (and suppressing indegree-1 outdegree-1 vertices) or of first refining a vertex (with outdegree greater than~2) and then removing an edge (and suppressing indegree-1 outdegree-1 vertices), see Figure~\ref{fig:cut}.

\begin{figure}
\centering
\includegraphics[scale=.8]{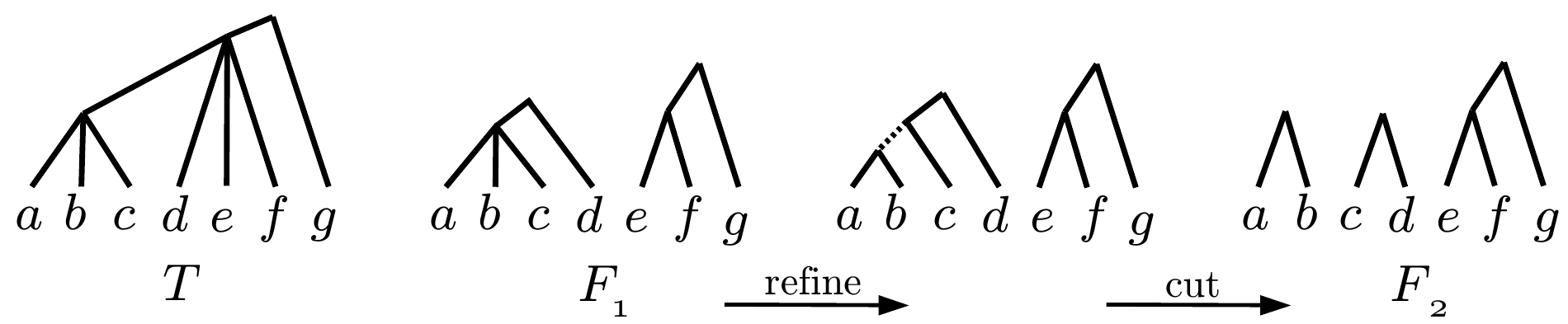}
\caption{A tree~$T$ and two forests~$F_1$ and~$F_2$ for~$T$. Forest~$F_2$ can be obtained from~$F_1$ by refining the parent of~$a,b$ and~$c$ and subsequently removing an edge and suppressing an indegree-1 outdegree-1 vertex.\label{fig:cut}}
\end{figure}

Let~$F_2$ be the forest obtained by cutting~$T_2$. In each iteration, we further cut~$F_2$ until at some point~$F_2$ becomes a forest also of~$T_1$. At that point, we have successfully obtained an agreement forest for~$T_1$ and~$T_2$ and we terminate the algorithm.

We describe an algorithm to determine which edges to cut by defining an iteration of the algorithm. Suppose that at the start of the iteration we have an (intermediate) forest~$F_2$. There are two main cases. In each case, the algorithm will make at most~4 cuts, and we will show that at least one of these cuts is unavoidable, thus showing that in each case a 4-approximation is attained.

Take an arbitrary internal vertex~$u$ of~$T_1$ with the property that all its children are leaves. Such a vertex clearly exists in any tree. Let~$C$ be the set of children of~$u$ in~$T_1$ and~$\bar{C}$ the set of all leaves that are not in~$C$.

First the algorithm checks the following three simple cases in the given order.

\smallskip

\noindent {\bf Case 0a.} There exist $c_1,c_2\in C$ that have a common parent in~$F_2$.

In this case, we collapse the subtree on $c_1$ and~$c_2$ to a single leaf in both~$T_1$ and~$F_2$. To be precise, we do the following in both~$T_1$ and~$F_2$. If~$c_1$ and~$c_2$ have no siblings, we delete~$c_1$ and~$c_2$ and label their former parent by~$\{c_1,c_2\}$. If~$c_1$ and~$c_2$ do have siblings, we delete~$c_1$ and replace label~$c_2$ by label~$\{c_1,c_2\}$.

\smallskip

\noindent {\bf Case 0b.} Some leaf~$c\in C$ is an isolated vertex in~$F_2$.

In this case, we remove~$c$ from both~$T_1$ and~$F_2$ and suppress any resulting outdegree-1 vertices. At the end, after recursively having computed an agreement forest, we add an isolated vertex for~$c$.

\smallskip

\noindent {\bf Case 0c.} The leaves in~$C$ are all in different components of~$F_2$. 

In this case, we remove all~$c\in C$ from both~$T_1$ and~$F_2$ and suppress any resulting outdegree-1 vertices. At the end, after recursively having computed an agreement forest, we add isolated vertices for all~$c\in C$.

\smallskip

Correctness of the procedure followed in the first two cases is obvious. To prove correctness in Case~0c, let~$F$ be an agreement forest of~$T_1$ and~$T_2$ that can be obtained by cutting~$F_2$. Since the leaves in~$C$ are all in different components of~$F_2$, they are all in different components of~$F$. Observe that any component of~$F$ that contains an element of~$C$ and an element of~$\bar{C}$ has to use the edge entering~$u$ in~$T_1$. It follows that at most one component of~$F$ contains an element of~$C$ and an element of~$\bar{C}$. Since the elements of~$C$ are all in different components of~$F$, it follows that at most one element from~$C$ is \emph{not} a singleton in~$F$. Hence, at most one of the cuts made in this step is avoidable. At least~2 cuts were made because $|C|\geq 2$ and no~$c\in C$ was already an isolated vertex in~$F_2$ by Case 0b. It follows that at least half of the cuts made were unavoidable.

If none of the above cases applies, the algorithm picks~$c_1,c_2$ from~$C$ in such a way that~$c_1$ and~$c_2$ are in the same component~$A_2$ of~$F_2$ and such that their lowest common ancestor in~$A_2$ is at maximum distance from the root of the component. Moreover, if there exists such a pair for which neither of~$c_1$ and~$c_2$ is a child of their lowest common ancestor, we pick such a pair first.

\smallskip

\noindent {\bf Case 1.} Neither of~$c_1$ and~$c_2$ is a child of their lowest common ancestor~$v$ in~$F_2$.

Let~$p_1$ be the parent of~$c_1$ and~$p_2$ the parent of~$c_2$ in~$F_2$. (We have $p_1\neq p_2$ because we are not in Case 0a.) Let~$S_1$ be the set of all leaves that are descendants of~$p_1$ except for~$c_1$. Similarly, let~$S_2$ be the set of all leaves that are descendants of~$p_2$ except for~$c_2$. Finally, let~$R$ be the set of other leaves of component $A_2$, i.e. leaves that are not in $S_1\cup S_2\cup\{c_1,c_2\}$. We cut~$A_2$ by creating separate components for~$c_1,c_2,S_1,S_2$ and~$R$, thus making four cuts, see Figure~\ref{fig:case1}.

\begin{figure}
    \centering
    \includegraphics[scale=.7]{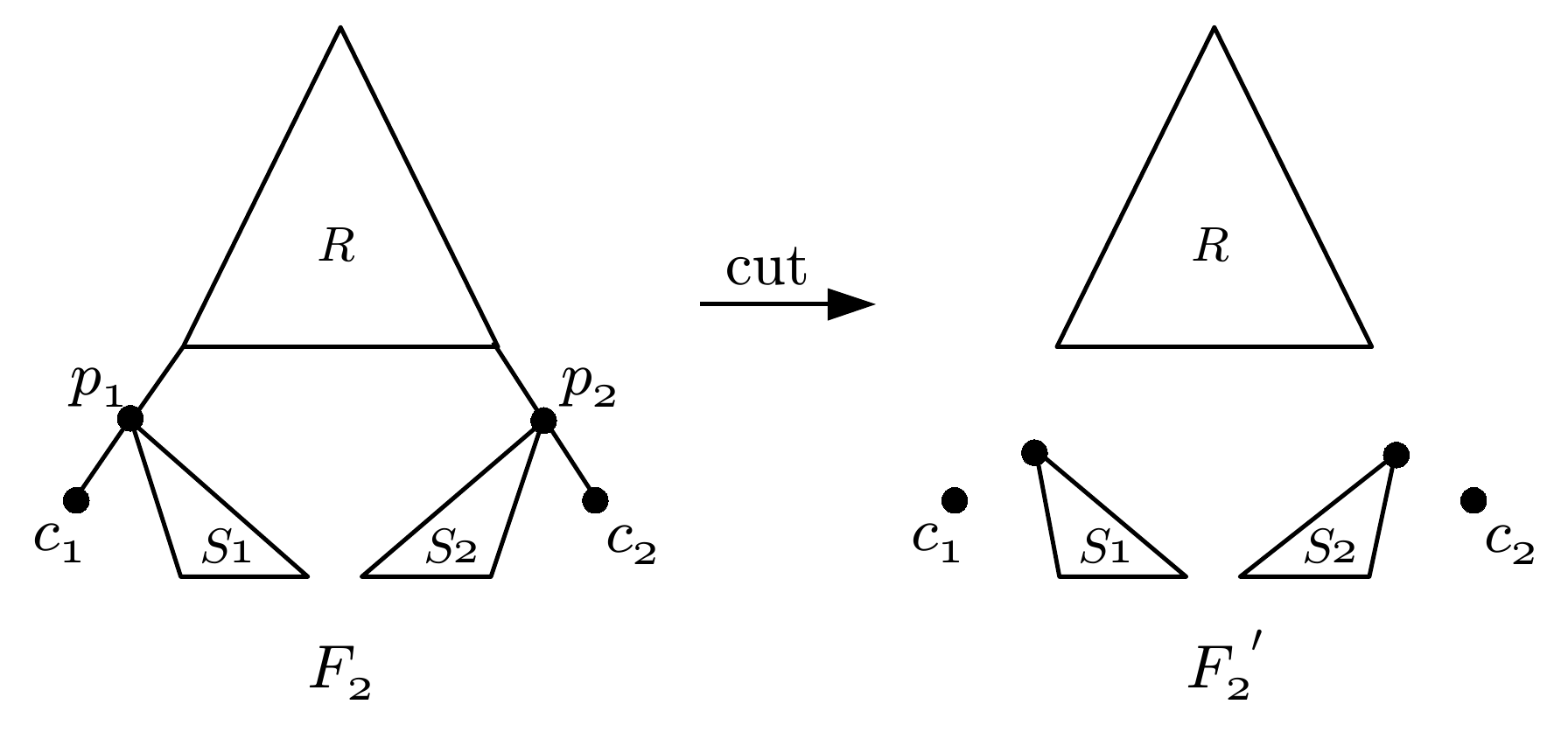}
    \caption{Case 1: none of~$c_1$ and~$c_2$ is a child of their lowest common ancestor. We cut~$F_2$ into~$F_2'$ by creating separate components for~$c_1,c_2,S_1,S_2$ and~$R$.\label{fig:case1}
    }
\end{figure}

We claim that at least one of these four cuts was unavoidable. Let~$F$ be an agreement forest of~$T_1$ and~$T_2$ with a minimum number of components. If any of~$c_1,c_2$ is a singleton in~$F$ then the cut that removed the edge entering that vertex was unavoidable. Hence, we assume that~$c_1$ and~$c_2$ are both non-singletons in~$F$. Suppose~$c_1$ and~$c_2$ are in the same component of~$F$. Because of the choice of~$c_1$ and~$c_2$ having their lowest common ancestor furthest from the root of the component,~$S_1,S_2\subset \bar{C}$; they contain no element of~$C$. Hence cutting off both $S_1$ and $S_2$ is unavoidable.

Hence, we assume that~$c_1$ and~$c_2$ are in different, non-singleton components of~$F$. As argued before, in justifying Case~0c, at most one of~$c_1$ and~$c_2$ can be in a component with elements from~$\bar{C}$. Moreover, as also argued before, $S_1,S_2\subset \bar{C}$. Therefore, at most one of~$c_1$ and~$c_2$ can be in a component with elements from $S_1\cup S_2$. W.l.o.g. suppose $c_1$ is not in a component with elements from $S_1\cup S_2$. Since~$c_1$ is not a singleton component, it is contained in a component which uses the edge of~$A_2$ entering~$p_1$ and the edge from~$p_1$ to~$c_1$, but none of the other edges leaving~$p_1$. Hence, the cut cutting off~$S_1$ is unavoidable. Similarly, if~$c_2$ is not in a component with elements from $S_1\cup S_2$, then the cut cutting off~$S_2$ is unavoidable. Thus, always at least one of the four cuts was unavoidable.

\smallskip

\noindent {\bf Case 2.} Either $c_1$ or $c_2$ is a child of their lowest common ancestor~$v$ in~$F_2$. 

Let~$c_1,c_2$ be any such pair. Let again~$p_1$ be the parent of~$c_1$ and~$p_2$ the parent of~$c_2$ in~$F_2$. Assume without loss of generality that~$p_1$ is the lowest common ancestor of~$c_1$ and~$c_2$. Let~$S_2$ contain all leaves that are descendants of~$p_2$ except for~$c_2$. Let~$S_1$ contain all leaves that are descendants of~$p_1$ except for $c_1,c_2$ and the leaves in~$S_2$. Let~$R$ contain all remaining leaves. We cut~$F_2$ by creating separate components for~$c_1,c_2,S_1,S_2$ and~$R$, thus making four cuts, see Figure~\ref{fig:case2}.

\begin{figure}
    \centering
    \includegraphics[scale=.7]{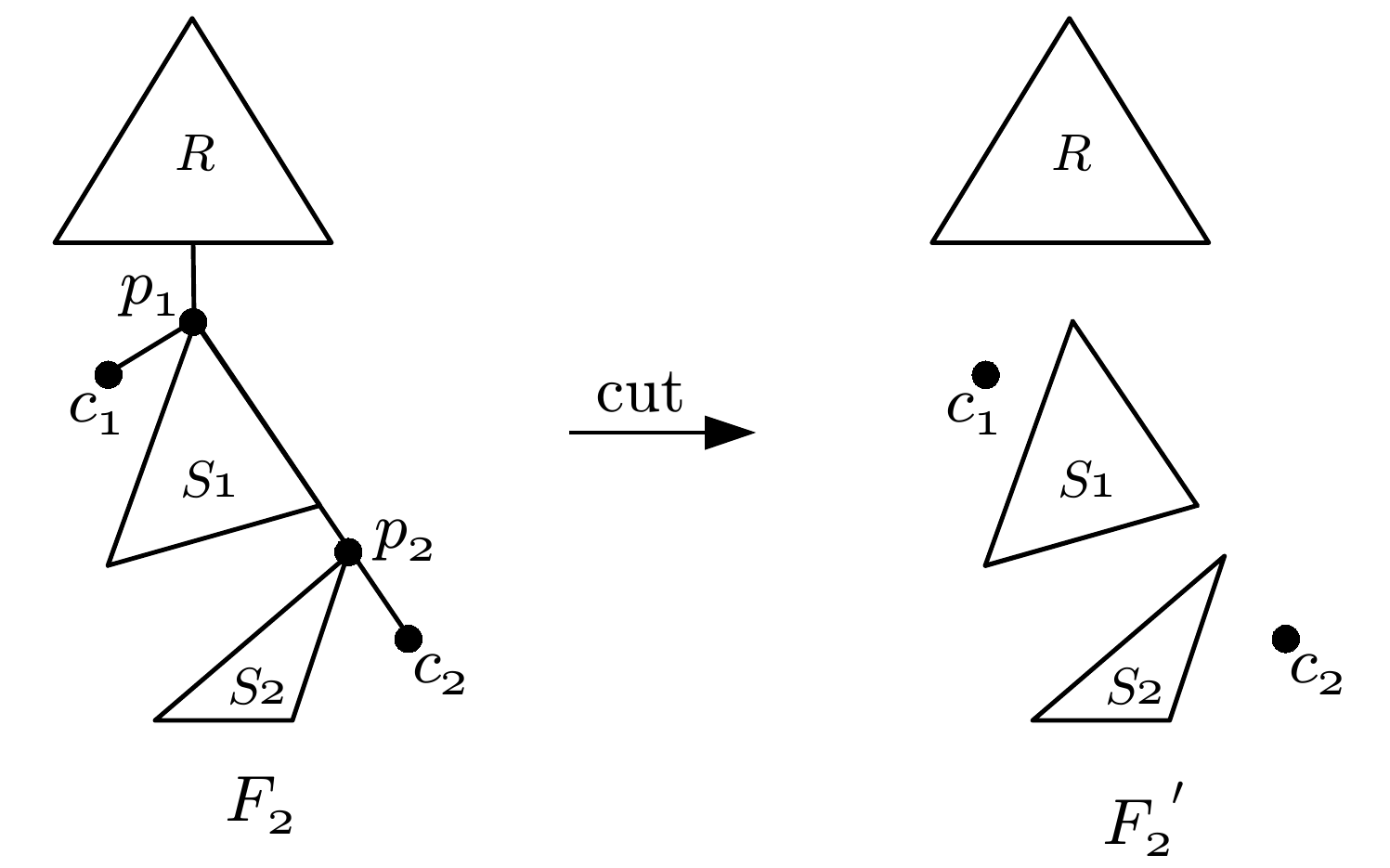}
    \caption{Case 2: $c_1$ is a child of the lowest common ancestor of~$c_1$ and~$c_2$. We cut~$F_2$ into~$F_2'$ by creating separate components for~$c_1,c_2,S_1,S_2$ and~$R$.
    \label{fig:case2}}
\end{figure}

We claim that also in this case at least one of the four cuts was unavoidable. Let~$F$ again be an agreement forest of~$T_1$ and~$T_2$ with a minimum number of components. We can argue as before that we can restrict attention to the situation in which~$c_1$ and~$c_2$ are non-singletons and belong to different components. It is also again true that~$S_1$ and~$S_2$ cannot contain elements from~$C$. To see this, first note that no element~$c_3$ of~$C\setminus\{c_1\}$ can be a child of~$p_1$ because then~$c_1,c_3\in C$ would have a common parent in~$F_2$, which is a Case~0a situation. Moreover, no~$c_3\in C\setminus\{c_1,c_2\}$ can be reached from~$p_1$ by a directed path with at least one internal vertex that does not belong to the path from $p_1$ to~$c_2$, because then $c_2,c_3$ would conform to Case~1. Finally, as before, no~$c_3\in C\setminus\{c_1,c_2\}$ can be reached from an internal vertex of the path from~$p_1$ to~$c_2$ because then~$c_3$ and~$c_2$ would have a lowest common ancestor further away from the root than~$p_1$. We conclude that~$S_1$ and~$S_2$ contain no elements from~$C$.

As in Case~1, it follows that at most one of~$c_1$ and~$c_2$ is in a component with elements from~$S_1\cup S_2$. Also, similar to the arguments in Case 1, if $c_1$ is not in a component with elements from~$S_1\cup S_2$ then that component uses the edge of~$A_2$ entering~$p_1$ and the edge from~$p_1$ to~$c_1$, but no other edges leaving~$p_1$. Hence, the cutting off $S_1\cup S_2\cup \{c_2\}$ is unavoidable. If~$c_2$ is not in a component with elements from~$S_1\cup S_2$, then cutting off~$S_2$ is unavoidable.

Hence, we conclude that in each case at least one of the four cuts is unavoidable, and the algorithm thus yields a 4-approximation.

\subsection{An FPT algorithm for nonbinary MAF}

In this section, we show that there exists an $O(4^k \poly(n))$ time algorithm for nonbinary \maf, i.e. we prove Theorem~\ref{thm:fpt}.

The algorithm follows the same ideas as the approximation algorithm in the proof of Theorem~\ref{thm:maf}. Cases~0a and 0b are executed in exactly the same way. In Case 0c, instead of removing all~$c\in C$, we pick~$c_1,c_2\in C$ arbitrarily and branch into two subproblems. In one subproblem~$c_1$ is removed and in the other subproblem~$c_2$ is removed. After recursively computing an agreement forest, the removed leaf is added as an isolated vertex. This step is correct since, by the proof of Theorem~\ref{thm:maf}, in any maximum agreement forest at least one of~$c_1$ and~$c_2$ is an isolated vertex. For, Cases~1 and~2, instead of making four cuts, we branch into four subproblems, one for each possible cut. By the proof of Theorem~\ref{thm:maf}, at least one of the four cuts is unavoidable, and hence at least one subproblem has the same optimum as the original problem.

It remains to analyse the running time. In each step we branch into at most four subproblems. For each subproblem we make one cut, and hence we reduce the parameter~$k$ by one. Therefore, at most~$4^k$ subproblems are created. For each subproblem, we need only time polynomial in~$n$. This concludes the proof.

\section{Conclusions and open problems}
We have given improved FPT and polynomial-time approximation algorithms for nonbinary MAF, and demonstrated that, as in the binary case, algorithms for MAF and DFVS can be combined to
yield nontrivial approximation guarantees for nonbinary MAAF. A number of interesting open problems remain. Firstly, the best known polynomial-time approximation algorithms for binary MAF
have a factor of 3, and for nonbinary this is now 4. Might it be that the binary and nonbinary variants are equally approximable, or is the nonbinary variant in some sense strictly more difficult
to approximate? For nonbinary MAAF we have shown how to achieve an approximation factor of $d(c+3)$, but for binary the corresponding expression is $d(c+1)$, this gap is also something
that needs to be explored.

On the software side, we have implemented both our MAF algorithms and made them publicly available~\cite{mafprogram}. We can report the following provisional performance results. The FPT algorithm solves instances with~$k\leq 14$ within a few minutes. The approximation algorithm solves instances with~$n\leq 500$ within seconds (and possibly larger instances too). The highest approximation-factor encountered on the inputs we tried was~$2.5$ (which raises the question whether the 4-approximation analysis given in this article can actually be sharpened). For nonbinary MAAF we have not yet implemented the $d(c+3)$ algorithm but are planning to do so. As in \cite{practicalCycleKiller}, it should be possible to obtain $d=1$ by using Integer Linear Programming (ILP) to solve DFVS exactly.

\section*{Acklowledgements}
We are grateful to Simone Linz for many useful discussions on the topic of this paper.

\bibliography{vanIerselKelkLekicStougie}{}
\bibliographystyle{plain}

\end{document}